\documentclass[draftclsnofoot,onecolumn]{IEEEtran}

\usepackage{balance}

\usepackage{cite}

\usepackage{amssymb}

\usepackage{amsmath}
\usepackage{amsthm}

\usepackage{algorithmic}

\usepackage{array}

\usepackage{bbm}

\newtheorem{theorem}{Theorem}
\newtheorem{lemma}{Lemma}
\newtheorem{corollary}{Corollary}

\newtheorem{remark}{Remark}
\newtheorem{claim}{Claim}

\DeclareMathAlphabet{\mathcurve}{OMS}{cmsy}{m}{n}

\newcommand{\define}{\overset{\triangle}{=}}
\newcommand{\argmin}[1]{\operatorname{arg}\min_{#1}}
\newcommand{\argmax}[1]{\operatorname{arg}\max_{#1}}
\newcommand{\E}[2]{\mathbb{E}_{#1}\left[#2\right]}

\DeclareMathOperator{\project}{Proj}
\newcommand{\proj}[2]{\project_#1\left(#2\right)}

\newcommand{\reals}{\mathbb{R}}
\newcommand{\dimcnt}{N}
\newcommand{\set}[1]{\mathcurve{#1}}
\newcommand{\Set}{\set{K}}
\newcommand{\subgrads}[2]{\partial #1(#2)}

\newcommand{\ZZ}{\mathbb{Z}}

\newcommand{\loss}{g}

\newcommand{\decision}{x}
\newcommand{\altdec}{v}

\newcommand{\xts}{\{\decision_t^*\}_{t=1}^T}
\newcommand{\dynamicregret}{R_T^d}

\usepackage{microtype}
\usepackage{graphicx}
\usepackage{subfigure}
\usepackage{booktabs}

\usepackage{hyperref}

\usepackage{algorithmic}

\begin{document}

\title{
	Universal Online Convex Optimization with Minimax Optimal $2^{nd}$-Order Dynamic Regret
}
\author{
	Hakan~Gokcesu,
	and~Suleyman~S.~Kozat
	\thanks{
		This study is partially supported by Turkcell Technology within the framework of 5G and Beyond Joint Graduate Support Programme coordinated by Information and Communication Technologies Authority.
		
		This work is also supported in part by Outstanding Researcher Programme Turkish Academy of Sciences. 
		
		H. Gokcesu and S. S. Kozat are with the Department of Electrical and Electronics Engineering, Bilkent University, Ankara, Turkey; e-mail: \{hgokcesu, kozat\}@ee.bilkent.edu.tr
		
		H. Gokcesu is also with Turkcell Teknoloji, Istanbul, Türkiye.
		
		S. S. Kozat is also with DataBoss A.S., Ankara, Türkiye.
	}
}

\maketitle

\begin{abstract}
	We introduce an online convex optimization algorithm which utilizes projected subgradient descent with optimal adaptive learning rates.
	Our method provides second-order minimax-optimal dynamic regret guarantee (i.e. dependent on the sum of squared subgradient norms) for a sequence of general convex functions, which may not have strong-convexity, smoothness, exp-concavity or even Lipschitz-continuity. The regret guarantee is against any comparator decision sequence with bounded path variation (i.e. sum of the distances between successive decisions). 
	We generate the lower bound of the worst-case second-order dynamic regret by incorporating actual subgradient norms. We show that this lower bound matches with our regret guarantee within a constant factor, which makes our algorithm minimax optimal. We also derive the extension for learning in each decision coordinate individually. We demonstrate how to best preserve our regret guarantee in a truly online manner, when the bound on path variation of the comparator sequence grows in time or the feedback regarding such bound arrives partially as time goes on. 
	We further build on our algorithm to eliminate the need of any knowledge on the comparator path variation, and provide minimax optimal second-order regret guarantees with no a priori information.
	Our approach can compete against all comparator sequences simultaneously (universally) in a minimax optimal manner, i.e. each regret guarantee depends on the respective comparator path variation. 
	We discuss modifications to our approach which address complexity reductions for time, computation and memory. We further improve our results by making the regret guarantees also dependent on comparator sets' diameters in addition to the respective path variations. 
\end{abstract}

\newpage
\section{Introduction}
\subsection{Preliminaries}
Convex programming, a major topic of online learning \cite{Shalev-Shwartz}, is extensively studied in the fields of 
automatic control, computational learning theory, signal processing and analysis.
In many tasks of optimization or prediction, the aim is to minimize some loss or error, many of which are convex functions, possibly time-varying. 
Examples include predictive control \cite{predictive_control}, network resource allocation \cite{Chen}, distributed agent optimization \cite{distributed_algorithms,Koppel}, fault diagnosis \cite{fault_diagnosis}, stochastic programming \cite{Liu}, adaptive filtering \cite{Liang}, beamforming \cite{Slavakis_adaptive} and classification \cite{Slavakis_online,gokcesu2021optimally}.

In the online optimization or learning setup, 
the convex objectives (i.e. the loss functions $f_t(\cdot)$) arrive sequentially. In particular, at each time $t$, we, the learner, or the controller, produce a decision $\decision_t$, and then, suffer the loss $f_t(\decision_t)$.
The role of a learning procedure, or the controlling agent,
is to choose $\decision_t$ so that the cumulative loss $\sum_{t=1}^{T} f_t(\decision_t)$ is minimized. 
To exemplify, in the 
sequential linear regression problem under 
absolute 
error, at each $t$, we decide on a parameter vector $\decision_t$, then, the nature reveals a feature vector $v_t$ and a desired output $d_t$, and we suffer the 
loss $f_t(\decision_t) = |\decision_t^\top v_t - d_t|$, 
which is a convex function with respect to 
$\decision_t$. 

In this work, we derive a learning algorithm applicable for any
convex loss function sequence $\{f_t(\cdot)\}_{t=1}^T$, 
that may not necessarily display additional desirable properties such as
strong-convexity, smoothness, exp-concavity or even Lipschitz-continuity, unlike \cite{Besbes,Mokhtari,Jadbabaie,Yang,Zhang_adaptive,Zhang_strongly,Zinkevich,gokcesu2021generalized}. 
The performance of such an online learning algorithm is traditionally evaluated relative to the best fixed decision in hindsight, 
e.g. the optimal fixed parameter vector. This evaluation metric is called
\emph{static regret}, which measures the difference between cumulative losses of our algorithm and the best fixed decision $\decision^*$. 
However, such a metric is insufficient in online dynamic scenarios where the best fixed decision itself performs poorly. To exemplify, consider an online learner to determine the optimal controller for a simple dynamic and time-varying zeroth-order (gain) discrete system with an open-loop (no feedback) circuit. Suppose the reference signal $r_t$ (desired output) is causal with the energy $\sum_{t=0}^{\infty} = E$. For some $\tau$, let the system gain be $S_0$ for $t<\tau$ and $S_1$ for $t\geq \tau$, and also, $\sum_{t=0}^{\tau-1} |r_t|^2 = E_0$ and $\sum_{t=\tau_e}^{\infty} |r_t|^2 = E_1$. The optimal static controller gain $C^*$ (lowest error energy) would be 
$
	C^*
	= \argmin{C} \left\{
		|1-S_0 C|^2 E_0 + |1-S_1 C|^2 E_1
	\right\}, 
$
as opposed to a time-varying controller with gain $C_0=S_0^{-1}$ for $t<\tau$ and $C_1=S_1^{-1}$ for $t\geq \tau$. Note that such an optimal dynamic controller would incur zero error, while the error of an optimal static controller could be very large, e.g., for $E_0=E_1$ and $S_1=-S_0$, the optimal (lowest energy) error signal would equal to the reference itself.

Henceforth, instead of the static regret, we measure performance with the generalized notion of \emph{dynamic regret}, 
allowing a time-varying comparator decision sequence $\{\decision_t^*\}_{t=1}^T$.
State-of-the-art algorithms in the literature achieve dynamic regret guarantees by imposing additional assumptions on the properties of $\{f_t(\cdot)\}_{t=1}^T$, which 
may not hold 
in real life scenarios, such as strong convexity \cite{Besbes,Mokhtari} (positive lower bound on the eigenvalues of Hessian matrix), Lipschitz-continuity  \cite{Besbes,Jadbabaie,Yang,Zhang_adaptive,Zhang_strongly,Zinkevich,prox_Villa_bounded_grad,prox_Bianchi_nonadaptive,TAC_OCO_nonadaptive,gokcesu2021regret,pmlr-v84-shen18a} (upper bound on subgradient norms), Lipschitz-smoothness \cite{Mokhtari,Yang,prox_Combettes_smooth} (upper-bounded eigenvalues of Hessian), and bounded temporal functional variations \cite{Zhang_strongly}. 

\newpage
More restrictive settings have also been investigated where a learner has access to the full information of the past functions 
\cite{Yang} or queries the subgradient of each $f_t(\cdot)$ at multiple points without incurring any additional losses \cite{Zhang_improved}. Such settings are also incompatible with many real life applications where the evaluation of each subgradient is costly (e.g. computationally) or even equivalent to making a decision, hence actually incurring a loss. 
In addition to the general online convex optimization, there exist dynamic studies for a specialized linear optimization problem, i.e. prediction with expert advice \cite{gokcesu2020recursive,Cesa-Bianchi_a,Herbster,Gokcesu2020AGO}. Moreover, adaptive regret guarantees dependent on subgradient norms are achieved for static regret \cite{Duchi,McMahan,Gupta}.

\subsection{Contributions}
As the first time in the literature, we introduce an efficient online projected subgradient descent algorithm for any sequence of convex loss functions,  
with dynamic regret guarantee of $$O\left(\sqrt{\left(D_* P_* + D^2\right)\sum_{t=1}^{T} \|g_t\|^2}\right)$$ where $D$ is the diameter of the projection set (e.g. feasible decision set) to which all $\decision_t$ and $\decision_t^*$ belong, $D_*$ is the maximum distance between any $\decision_t^*$,
i.e. 
$D_* = \max_{(t,\tau)} \|\decision_t^* - \decision_\tau^*\|,
$ 
and $P_*$ is 
the path variation measuring the complexity of our competition, 
i.e. 
$P_*=\sum_{t=1}^{T-1} \|\decision_{t+1}^* - \decision_t^*\|$.
This guarantee is simultaneously achieved against all comparators, universally. 

Even though it is varying against different comparators (proportional to $P_*$ and $D_*$), it is shown to be minimax optimal for each comparator individually. This optimality is demonstrated by showing that the regret guarantees match the respective worst-case dynamic regret lower bounds with adaptive dependency on the subgradient norms instead of some preconceived bounds.
Furthermore, via certain extensions, our complexities (computation, time, memory) can be efficiently reduced with certain (generally acceptable) trade-offs.
The development of our results is summarized as follows.

We derive an algorithm which achieves minimax optimal guarantees against any comparator sequence $\{\decision_t^*\}_{t=1}^T$ where $\sum_{t=1}^{T-1} \|\decision_{t+1}^* - \decision_t^*\| \leq P$ and $P$ is known a priori by the algorithm.
We enhance our method with the ability to incorporate a time-growing $P$, i.e. $P(T)$, which is reasonable as $\sum_{t=1}^{T-1} \|\decision_{t+1}^* - \decision_t^*\|$ is nondecreasing with $T$. Then, we alternatively suppose the knowledge of $P$ arrives partially throughout optimization which is a less restrictive setting.

We introduce our universal method which assumes no knowledge regarding $P$, i.e. we compete against any comparator sequence $\{\decision_t^*\}_{t=1}^T$, and each competition is individually minimax optimal where the regret bounds depend on each comparator sequence separately via $$P_* = \sum_{t=1}^{T-1} \|\decision_{t+1}^* - \decision_t^*\|, \quad D_* = \max_{(t,\tau): 1\leq t<\tau\leq T} \|\decision_t^* - \decision_\tau^*\|.$$ 

Our work on universality, e.g. parameter-free algorithm (no prior setting of $P_*$ and $D_*$), is the most novel, eliminating the need of any prior (hindsight) information and producing a truly online minimax optimal approach as a first in the literature.

\subsection{Organization}
The work is constructed as follows. In Section \ref{sec:oco}, we formally describe the problem. In Section \ref{sec:adagrad}, we introduce our base algorithm and derive its dynamic regret upper bound for a known fixed $P$, and show its simplicity and superiority against some existing approaches. Following that, we demonstrate 
the optimality of our algorithm by introducing a worst-case dynamic regret lower bound that matches our regret guarantee up to a constant factor
and investigate the cases when path variation is coordinate-wise separable, knowledge on $P$ partially arrives and known $P$ grows in time, respectively. In Section \ref{sec:no_know}, we adapt our algorithm so that no prior knowledge of $P$ is required, and show its capability to provide universal regret guarantees. Specifically, we demonstrate the ability to obtain guarantees for each possible path variation and effective diameters implied by each comparator sequence in a joint and universal manner. In Section \ref{sec:simulation}, we demonstrate our performance by simulating an optimization task. We conclude with Section \ref{sec:conclusion}. Many proofs of the analyses are given in the appendix.

\newpage
\section{Problem Description} \label{sec:oco}
We have a sequence of convex functions $f_t: \Set \rightarrow \reals$ for discrete times $t \geq 1$, where $\Set$ is a convex\footnote{For all $\decision,\altdec \in \Set$, $(\lambda \decision + (1-\lambda)\altdec) \in \Set$ for any $0\leq \lambda\leq 1$.}, closed and bounded subset of $\reals^{\dimcnt}$.
Each $\decision \in \Set$ is a column vector, $\decision^\top$ is its transpose, $\decision^\top \altdec$ is its inner product with $\altdec \in \Set$, and $\|\decision\| = \sqrt{\decision^\top \decision}$ is its Euclidean norm. The projection $\proj{\Set}{\decision}$ solves $\argmin{\altdec \in \Set} \|\decision-\altdec\|^2$, a relatively simple computation when $\Set$ is a hyper-ellipsoid/rectangle.

Convexity of each $f_t(\cdot)$ implies the first-order relation
\begin{equation} \label{eq:convexity}
	f_t(\decision) - f_t(\altdec) \leq g_t^\top (\decision - \altdec)
\end{equation}
for every pair $\decision,\altdec \in \Set$ and every subgradient 
$g_t \in \subgrads{f_t}{\decision}$.

Then, the \emph{dynamic regret}, denoted as $\dynamicregret$, is defined as
\begin{equation} \label{eq:linearized_regret}
	\dynamicregret
	\define \sum_{t=1}^T f_t(\decision_t) - f_t(\decision_t^*) \leq \sum_{t=1}^T g_t^\top (\decision_t - \decision_t^*),
\end{equation}
where $\{\decision_t\}_{t=1}^T$ and $\{\decision_t^*\}_{t=1}^T$ are the algorithm's and comparator's (best) decision sequences, respectively, and the inequality comes from \eqref{eq:convexity} for any $g_t \in \subgrads{f_t}{\decision_t}$.
This is a tight bound when 
$f_t(\cdot)$ are only known to be convex and holds with equality for linear functions $f_t(\decision) = g_t^\top \decision$.

Note that $\dynamicregret$ cannot be bounded in a nontrivial manner, i.e. sublinear $o(T)$ bounds, without some restrictions on $\xts$, which becomes apparent due to the worst-case (adversarial) regret lower bounds we discuss in Section \ref{sec:lower_bound}.
Thus, we 
control the complexity of our competition class by considering
$\xts$ such that $\xts \in \Omega_T^P$ where
\begin{equation*} \label{eq:restricted_best_space}
	\Omega_T^P = \left\{\xts : \decision_t^* \in \Set, \sum_{t=1}^{T-1} \|\decision_{t+1}^* - \decision_t^*\| \leq P\right\},
\end{equation*}
with $P \leq D(T-1)$ (naturally) for $D = \sup_{\decision,\altdec \in \Set} \|\decision-\altdec\|$, which is the diameter of $\Set$. The parameter $P$ of the competition class $\Omega_T^P$ is also called \emph{path variation} \cite{Yang}. This class generalizes the special case $P=0$ corresponding to the \emph{static regret} (best fixed decision).

\newpage
\section{
	Online Projected Subgradient Descent} \label{sec:adagrad}
\renewcommand{\figurename}{Algorithm }
\begin{figure}
	\centering
	\caption{Online Subgradient Descent with Projection}\label{algorithm:gradient_descent}
	\begin{algorithmic}[1]
		\REQUIRE $\decision_1 \in \Set$, $g_t \in \subgrads{f_t}{\decision_t}$ for $t\geq1$.
		\ENSURE $\decision_t \in \Set$ for $t \geq 2$.
		\STATE Initialize $t=1$.
		\WHILE{$g_t$ is the zero-vector}\label{0-vec}
		\STATE $\decision_{t+1} = \decision_t$.
		\STATE $t \leftarrow t+1$.
		\ENDWHILE
		\WHILE{$g_t$ received \textbf{and} not terminated by the user}
		\IF{$g_t$ is not the zero-vector}
		\STATE Decide $\eta_t$.
		\STATE Set $\decision_{t+1} = \proj{\Set}{\decision_t - \eta_t g_t}$.
		\STATE $t \leftarrow t+1$.
		\ENDIF
		\ENDWHILE
	\end{algorithmic}
\end{figure}
We use online subgradient descent with projection as shown in \figurename \ref{algorithm:gradient_descent} to update our decisions. It is a variant of the proximal descent method when we do not assume any underlying functional bias and it is still commonly studied \cite{TAC_PGD}. Given the feasible decision set $\Set$, and Euclidean projection operator $\proj{\Set}{\cdot}$, the utilized update is such that
\begin{equation*}
	\decision_{t+1} = \proj{\Set}{\decision_t - \eta_t g_t},
\end{equation*}
for learning rates (step sizes) $\eta_t$ and subgradients $g_t$.
In the following theorem, we 
investigate a general regret guarantee for a nonincreasing positive learning rate sequence. Then, we show how to sequentially select ideal learning rates.

\begin{theorem} \label{theorem:gradient_descent_regret}
If we run 
\figurename \ref{algorithm:gradient_descent} with a nonincreasing positive $\eta_t$ sequence, the
dynamic regret can be bounded as,
\begin{displaymath}
\dynamicregret \leq \frac{D^2(P/D+1/2)}{\eta_T} + \sum_{t=t_0}^T \frac{\eta_t}{2} \|g_t\|^2
\end{displaymath}
where $D$ is the diameter of $\Set$, $P \geq \sum_{t=1}^{T-1} \|\decision_{t+1}^* - \decision_t^*\|$ and $\|g_\tau\|=0$ for $\tau < t_0$. 
\end{theorem}

\newpage
\subsection{Optimal Learning Rates}
We first define the following quantities, which will be used to determine the optimal subgradient descent learning rates at each time, i.e. $\{\eta_t\}_{t=1}^T$. For $t\geq1$, with $G_0 
= 0$,
\begin{align} \label{eq:root_sum_g_t^2}
G_t 
\define \sqrt{G_{t-1}^2 + \|g_{t}\|^2}
= \sqrt{\sum_{\tau=1}^t \|g_\tau\|^2}.
\end{align}

\begin{corollary} \label{corollary:gradient_descent_fixed_eta}
If we were to use the optimal constant learning rate $\eta_t = \eta^*$, which minimizes the right-hand side of the guarantee in Theorem \ref{theorem:gradient_descent_regret}, it yields $\eta_t = D\sqrt{1+2P/D} G_T^{-1}$ and
\begin{displaymath}
\dynamicregret \leq D\sqrt{1+2P/D}G_T,
\end{displaymath}
for $P \geq \sum_{t=1}^{T-1} \|\decision_{t+1}^* - \decision_t^*\| \geq 0$.
\end{corollary}
\begin{proof}
If $G_T = 0$, 
$R_T \leq \sum_{t=1}^{T}0^\top (\decision_t - \decision_t^*) = 0 = D G_T$.
If $G_T > 0$ instead, it follows from Theorem \ref{theorem:gradient_descent_regret}.
\end{proof}

The optimal constant rate $\eta_t = \eta^*$ requires the future information 
$\|g_\tau\|$ for $\tau > t$.
We now present an adaptive causal learning 
scheme 
using only the past information.
\begin{corollary} \label{corollary:gradient_descent_adaptive_eta}
If we use $\eta_t = DG_t^{-1}\sqrt{\hat{P}/D+1/2}$, we obtain the following
regret guarantee
$$		\dynamicregret \leq DG_T\left(
\frac{P/D+1/2}{\sqrt{\hat{P}/D + 1/2}} + \sqrt{\hat{P}/D + 1/2}
\right),
$$
which simplifies into
\begin{displaymath}
\dynamicregret
\leq 2DG_T\sqrt{\hat{P}/D+1/2} ,
\end{displaymath}
if $\hat{P} \geq P$, and is optimal (minimized) for $\hat{P} = P$.
\end{corollary}

The regret guarantees in Corollaries \ref{corollary:gradient_descent_fixed_eta} and \ref{corollary:gradient_descent_adaptive_eta} are -up to a constant factor- equivalent to the minimax regret lower bound, yet to be shown in Theorem \ref{theorem:lower_bound_sum}.

\newpage
\subsection{Comparing with Projections using Self Outer Products}
The adaptive regret guarantee in Corollary \ref{corollary:gradient_descent_adaptive_eta} for $P=0$ outperforms $O\left(tr\left(A_T\right)\right)$, which is the static regret guarantee provided by the projected normalized sub-gradient descent using the root of self outer products sum $A_T = \sqrt{\sum_{t=1}^T g_t g_t^T}$, e.g. Ada-Grad with full matrix divergences \cite{Duchi}. 

\begin{claim} \label{claim}
	Given a vector sequence $g_t$ for $1 \leq t \leq T$,
	\begin{displaymath}
		\sqrt{\sum_{t=1}^T \|g_t\|^2}
		\leq
		tr\left(\sqrt{\sum_{t=1}^T g_t g_t^T}\right),
	\end{displaymath}
	where the square root of the outer products sum on the right-hand side corresponds to its unique positive semidefinite principal square root, 
	and the trace operation $tr(\cdot)$ takes the sum of elements on the main diagonal.
\end{claim}

There is a discrepancy between the left and right sides of Claim \ref{claim} which could rise up to a multiplicative term of $\sqrt{\dimcnt}$ (when the eigenvalues $\lambda_i$ of the outer product sum are similar) -where $\dimcnt$ is the dimension of our decision set $\Set$- even though the algorithm we present is more efficient, i.e. at each time $t$ our algorithm computes the inner product $g_t^T g_t$ while algorithms with full divergences compute the outer product $g_t g_t^T$. 

\subsection{Minimax Dynamic Regret Lower Bounds} \label{sec:lower_bound}
Given any online, i.e. sequential and causal, learning algorithm, we 
show that there exists a sequence of $\{f_t(\cdot)\}_{t=1}^T$ such that the regret in \eqref{eq:linearized_regret} is lower bounded 
as follows.

\begin{theorem} \label{theorem:lower_bound_sum}
	For 
	$\sum_{t=1}^{T} \|g_t\|^2 = G_T^2$, for any causal algorithm, there exists a $\{f_t(\cdot)\}_{t=1}^T$ sequence such that this algorithm may incur a worst-case dynamic regret $\overline{\dynamicregret}$ as
	\begin{displaymath}
		\overline{\dynamicregret} \geq (DG_T/2\sqrt{2}) \sqrt{\lfloor P/D \rfloor+1} ,
	\end{displaymath}
	where $D$ is the diameter of $\Set$ and $P \geq \sum_{t=1}^{T-1} \|\decision_{t+1}^* - \decision_t^*\|$.
\end{theorem}

The lower bound in Theorem \ref{theorem:lower_bound_sum} matches the upper bound in Corollary \ref{corollary:gradient_descent_adaptive_eta}, within a factor, 
thus our adaptive algorithm is optimal in a strong (second-order) minimax sense. 

In the following corollary, we also discuss how the lower bound behaves for a uniform gradient norm constraint, i.e. $\exists L$ such that $\|g_t\| \leq L$ for all $t$. We generate a zeroth-order bound and replace the need of a very mild assumption for the worst-case from before, requiring the sums of squared norms from consecutive time segments $[t_{k-1}+1,t_k]$, i.e. $\sum_{t=t_{k-1}+1}^{t_k} \|g_t\|^2$, to be of the same order as each other, with Lipschitz continuity.

\begin{corollary} \label{corollary:lower_bound_max}
	For 
	$\max_{1\leq t\leq T} \|g_t\| \leq L$ instead,
	\begin{displaymath}
		\dynamicregret \geq (D/4) \sqrt{\lfloor P/D \rfloor+1} L\sqrt{T}.
	\end{displaymath}
\end{corollary}

\newpage
\subsection{Separate Step Sizes for Each Decision Entry}
We extend our results to the case when each entry (or block) employs independent learning rates. For that, we denote the $i^{th}$ entry block of decisions $\decision, \decision_t \in \Set$ as $\decision_{(i)}, \decision_{t,i}$, respectively.
\begin{remark} \label{remark:separate_eta}
	After rewriting \eqref{eq:linearized_regret} as a sum over coordinate blocks and applying Corollary \ref{corollary:gradient_descent_adaptive_eta} for 
	each block, 
	we get
	\begin{equation*} \label{eq:gradient_descent_regret_g_{t,i}}
		R_T \leq \sum_{i=1}^\dimcnt 
		2D_i\sqrt{P_i/D_i+1/2} \sqrt{\sum_{t=1}^{T} \|g_{t,i}\|^2},
	\end{equation*}
	for $D_i = \sup_{\decision,\altdec \in \Set} \|\decision_{(i)}-\altdec_{(i)}\|$, $P_i \geq \sum_{t=1}^{T-1} \|\decision_{t+1,i}^* - \decision_{t,i}^*\|$. 
\end{remark}

The minimax optimality is preserved if the decision set is separable, i.e. 
$$\Set = \Set_1 \times \ldots \times \Set_M,$$ 
where the number of coordinate blocks is $M$, and 
$$(\decision \in \Set) \iff (\forall i\in \{1,\ldots,M\}, \decision_i \in \Set_i).$$

In the following subsections, we investigate various scenarios for path variations and gradually remove the need to pre-set the quantity $P$.

\newpage
\subsection{Partial Information on Path Variation} \label{sec:partial_info}
Here, we investigate the scenario where path variation constraint $P$ is not a priori known but is revealed gradually. Consider that, following $t_{k-1}$ for $1\leq k\leq K$, we receive a hint that until (including) $t_{k}$, we have a path variation $P_k$ for 
$\{\decision_t^*\}_{t_{k-1}+1}^{t_k}$. The incurred regret can be upper bounded 
as shown in the following.

\begin{theorem} \label{theorem:with_resets}
Assume for $1\leq k\leq K$, each $P_k$ corresponding to the best decision sequence segments $\{\decision_t^*\}_{t_{k-1}+1}^{t_k}$ with $t_0=0$ and $t_K=T$, is known at the latest following $(t_{k-1})^{th}$ round. Under such conditions, if we reset Algorithm \ref{algorithm:gradient_descent} following times $t_{k-1}$ and use the adaptive step sizes in Corollary \ref{corollary:gradient_descent_adaptive_eta} with $\hat{P} = P_k$ for $t_{k-1}< t\leq t_k$, we 
upper bound the incurred 
regret as
\begin{align*}
\dynamicregret 
\leq \sum_{k=1}^{K} 2D\sqrt{P_k/D+1/2} \sqrt{\sum_{t=t_{k-1}+1}^{t_k} \|g_t\|^2}
\end{align*}
where $D$ is the feasible set diameter. Assuming $K$ is relatively low, i.e. $K \leq C \sum_{k=1}^{K} P_k/D$ for some constant $C$, which is reasonable to assume as even the variation between successive ``best" decisions, i.e. $\|\decision_{t+1}^* - \decision_t^*\|$, can be up to $D$, or $K$ is finite, this upper bound is minimax optimal within a constant factor 
in accordance with the 
worst-case lower bound in Theorem \ref{theorem:lower_bound_sum}.
\end{theorem}

\begin{remark}
If we consider the case where a path variation feedback $P_k$ can also refer to the past such that the feedback $P_k$ arriving after $t_{k-1}$ does not bound the path variation for $t_{k-1}< t\leq t_k$ but for $\tau_k< t\leq t_k$ with $\tau_k \leq t_{k-1}$. Then, the $k^{th}$ run of the algorithm after step size resetting can utilize $\hat{P} = \min\{P_k,D(t_k-t_{k-1}-1)\}$ where the second argument of
$\min$ 
arises from the utmost limit of path variation during the new 
segment $t_{k-1}<t \leq t_k$. Additionally, whenever there
are
time segments with no path variation feedback revealed in preparation for the corresponding run (e.g. $\tau_k> t_{k-1}$), then these gaps 
utilize 
the utmost limits, 
i.e., during $t_{k-1}< t\leq t_k$, use $\hat{P} = P_k + D(\tau_k-t_{k-1}-1)$.
\end{remark}

\newpage
\subsection{Best Sequence Constraint Grows in Time} \label{sec:time_grow}
We now investigate how to handle a time-increasing $P$, i.e. a function $P(\cdot)$ such that $$P(t) \geq \sum_{\tau=1}^{t-1} \|\decision_{\tau+1}^* - \decision_\tau^*\|,$$ meaning $P(T)$ is the upper bound of the best sequence path variation for the optimization of duration $T$. We employ a sort of ``doubling trick" with some knowledge on $P(\cdot)$. For $k\geq 1$, we identify $t_k = \argmax{t: P(t)\leq D(2^{k-1}-1)} P(t)$. To exemplify, considering the natural bound $P(t) \leq D(t-1)$, and then, suppose $P(t) = D(t^p-1)$ for some real number $0\leq p\leq 1$. This means $t_k = \lfloor 2^{(k-1)/p} \rfloor$, where $\lfloor \cdot \rfloor$ is the flooring function. 

Note that we do not 
need 
to access $P(\cdot)$ in full a priori, i.e. for all integers $t\geq 1$. The identifications of $t_k$ for increasing $k$ can be done iteratively following $t_{k-1}$. 
Then, after each $t_{k-1}$, we reset Algorithm \ref{algorithm:gradient_descent} with learning rates selected according to Corollary \ref{corollary:gradient_descent_adaptive_eta} with path variation set as $P=D (2^{k-1}-1)$ for that run. When the differential $P(t_{k-1}+1) - P(t_{k-1})$ is too great, the duration of $k^{th}$ run can even be $0$, i.e. that specific run is effectively skipped.
This scheme results in the following theorem. 
\begin{theorem} \label{theorem:constraint_grows_in_time}
When the best decision sequence path variation for a $T$-length optimization is bounded by some nondecreasing $P(T) \geq \sum_{t=1}^{T-1} \|\decision_{t+1}^* - \decision_t^*\|$, 
Algorithm \ref{algorithm:gradient_descent} is reset following rounds $t_{k-1}$, with $t_0=0$. The new run lasts until (including) $t_k = \argmax{t: P(t)\leq D(2^{k-1}-1)} P(t)$, and this run sets the path variation as 
$P_k = D(2^{k-1}-1)$. 
By using the path variation in full, from 
$t=1$ up to $t=t_k$, our algorithm considers the possibility where we compete against a best decision sequence which had low to none path variation till now, i.e. $t_{k-1}$. 
The resulting 
regret is 
\begin{equation*}
\dynamicregret \leq 4D \sqrt{\frac{P(T)}{D} + \frac{6-K}{8}} \sqrt{\sum_{t=1}^{T} \|g_t\|^2},
\end{equation*}
where $K = \min_{x\in \ZZ: P(T)\leq D(2^{x-1}-1)} x$, and $D$ 
is 
the diameter of 
decision set $\Set$. This 
guarantee is minimax optimal within a constant factor 
in accordance with Theorem \ref{theorem:lower_bound_sum}.
\end{theorem}

If we cannot even query $P(\cdot)$ every round, 
we can employ an exponential search method. Following remark explains this.

\begin{remark}
Starting with the knowledge of some $t_{k-1}$, query $P(t_{k-1}2^{i})$ for $i\geq 1$ until $P(t_{k-1}2^{m}) > D(2^{k-1}-1)$ is identified, ensuring that $t_{k-1}2^{m} > t_k$.  Then, we identify $t_k \in (t_{k-1},t_{k-1}2^{m})$ via binary search, such that we have $t_k = \argmax{t: P(t)\leq D(2^{k-1}-1)} P(t)$. If the querying of $P(\cdot)$ is further restricted, e.g. it can be queried for both low number of times and only in increasing arguments, we can employ the doubling trick on $t_k$, e.g. $t_k = 2^{k}$, with the mild assumption of a near subadditive property such that we have $P(T_1+T_2) \leq P(T_1)+P(T_2+1)$, which is somewhat reasonable as $P(T) \leq D(T-1)$ (i.e., $P(\cdot)$ is at most linear). 
\end{remark}

\newpage
\section{Universal Regret for Varying Path Variations} \label{sec:no_know}
In this section, we build upon our previous findings and prepare parallel running algorithms as agents for an eventual mixing scheme (i.e., prediction with expert advice) to circumvent the need of any knowledge on the path variation $P$. As the time horizon $T$ grows, the number of such agents can be limited to $O(\log T)$ thanks to Corollary \ref{corollary:gradient_descent_adaptive_eta} by using $P_m = D(2^{m-1}-1)$. Then, the regret 
of $m^{th}$ running agent, which sets the path variation $P_m$, with $m<\log(2T)$, is bounded as follows.
\begin{corollary} \label{corollary:mth_running_algorithm}
Consider runs of Algorithm \ref{algorithm:gradient_descent}, indexed as $m^{th}$, with the resetting as in Theorem \ref{theorem:with_resets}. These runs have no access to any path variation knowledge but 
use predefined 
values 
$P_k$ 
following set reset rounds $t_{k-1}$, i.e. $m^{th}$ run uses $t_k = 2^k-1$ and $P_k 
= D(2^{k-1} - 1)$ for $1\leq k\leq m-1$. The last pair $(t_m,P_m)$ for run $m$ is set as $t_m = T$ (i.e. no more reset) and $P_m = D(2^{m-1} - 1)$. In accordance with Theorem \ref{theorem:with_resets} -and Corollary \ref{corollary:gradient_descent_adaptive_eta}-, each run incurs 
\begin{displaymath}
\dynamicregret(m) 
\leq 2D\sqrt{2^m - m/2 - 1} G_T
\end{displaymath}
whenever the actual overall path variation $P$ during the whole optimization is such that $P \leq P_m = D(2^{m-1}-1)$. 
The regret bound for $P > P_m = D(2^{m-1}-1)$ is not needed for what follows. 
Also, $2^{m^*-2} < P/D + 1 \leq 2^{m^*-1}$ for an $m^*$ and  
$R_T^d(m^*)$ is minimax optimal up to some constant factor. 
\end{corollary}
\begin{proof}
Similar to Theorem \ref{theorem:constraint_grows_in_time}.
\end{proof}

\subsection{Explicit Linear Optimization}
Here, we explain averting the need of subgradient evaluation for each agent separately running Algorithm \ref{algorithm:gradient_descent}. Similar to \cite{Zhang_adaptive, NEURIPS2018_10a5ab2d}, we show that the regret can be optimally bounded with a single subgradient evaluated at each time $t$ by considering the problem as purely linear. Using \eqref{eq:linearized_regret}, we separate the bound into two sums where each is to be optimized 
by alternating between subgradient descents and expert mixtures. 
Thus,
\begin{align} \label{eq:separated_regret}
&\dynamicregret
\leq R_T^{l,e}(m) + R_T^{l,d}(m),  
\text{ for all $m$, with}
\\
&R_T^{l,e}(m) \define \sum_{t=1}^T g_t^\top (\decision_t - \decision_t^m)
,\;
\nonumber
R_T^{l,d}(m) \define \sum_{t=1}^T g_t^\top (\decision_t^m - \decision_t^*),
\nonumber
\end{align}
where $R_T^{l,e}(m)$ denotes the static regret of the expert mixture scheme against the loss incurred by some $m^{th}$ running Algorithm \ref{algorithm:gradient_descent} 
and $R_T^{l,d}(m)$ denotes the linearized upper bound of the dynamic regret of the $m^{th}$ running Algorithm \ref{algorithm:gradient_descent}, all in accordance with Corollary \ref{corollary:mth_running_algorithm}, i.e. $\dynamicregret(m)$. 

Hence, the alternating optimization will work as follows. At each time $t$, during the first (initial) stage, $R_T^{l,d}(m)$ is optimized for all 
agents in parallel, each indexed by a different $m$, with each such agent producing a decision $\decision_{t}^m$, and, during the second (final) stage, each $R_T^{l,e}(m)$ is optimized by 
combining all 
$\decision_{t}^m$, thus producing the final decision $\decision_t$.
All that remains is to construct the way to mix the decisions of these parallel running agents. 

\newpage
\subsection{Recursive Mixture}
We consider a recursive expert mixture similar to \cite{gokcesu2020recursive}, where the base mixing algorithm is from \cite{Cesa-Bianchi_improved} with 2 experts. The reason for this choice of mixture is twofold. First, since our setting is truly online, i.e. $T$ is not predetermined and unbounded, a straightforward application of the mixing method from \cite{Cesa-Bianchi_improved} is not possible. It has to be supplemented with some variant of the doubling trick. Adding the fact that, the regret redundancy due to a regular expert mixture has an additional multiplier of $O(\sqrt{\log M})$ where $M$, the number of experts/agents (i.e., parallel runs of Algorithm \ref{algorithm:gradient_descent}) in the mixture, also grows to $O(\log T)$ in an unbounded manner, the regret from such an expert mixture would gather a possible multiplicative optimality gap which grows with time, i.e. not finite and thus invalidating the minimax optimality claim.

We can circumvent this by obtaining differing mixture regret redundancies when compared to each expert individually, i.e. the regret guarantees against the optimal and any other experts may be different. Since the guarantees from Algorithm \ref{algorithm:gradient_descent} differs for each expert, we want to carefully create this discrepancy in the mixture redundancies against each expert so that they are upper-bounded by the regret from their Algorithm \ref{algorithm:gradient_descent} counterparts within a constant factor, thus preserving the minimax optimality claim. 

\renewcommand{\figurename}{Figure }
\begin{figure}
	\caption{Recursive Mixture Illustration}
	\centering
	\includegraphics[width=0.815\linewidth]{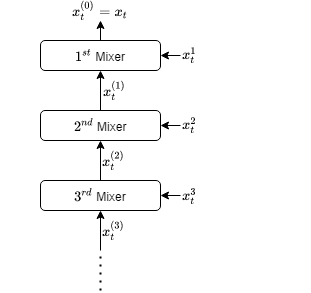}
	\label{fig:recursive}
\end{figure}

The aforementioned discrepancy in mixture redundancies is achieved exactly by the recursive application of \cite{Cesa-Bianchi_improved}, where each mixing is done at the branching-out locations in the growing ensemble of the experts in a skewed tree form, as illustrated in Figure \ref{fig:recursive}. The regret redundancy components are computed after the following separation of $R_T^{l,e}(m)$:

\begin{align}
R_T^{l,e}(m) 
= &\sum_{k=1}^{m-1} \left(\sum_{t=1}^T g_t^\top (\decision_t^{(k-1)} - \decision_t^{(k)})\right)
+ \sum_{t=1}^T g_t^\top (\decision_t^{(m-1)} - \decision_t^m),
\label{eq:expert-split}
\end{align}
where $\decision_t^{(k-1)}$ is the decision obtained by recursive mixing of another mixture output $\decision_t^{(k)}$ and $\decision_t^{k}$ from $k^{th}$ run of Algorithm \ref{algorithm:gradient_descent}; $\decision_t^{(0)} = \decision_t$, i.e. the final decision; and the sum $\sum_{k=1}^{m-1} (\cdot)$ is ignored for $m=1$. Each mixing is in effect starting at their branch-out times, until then the mixing is pointless since the two experts output the same decision. Consequently, for competing against the $m^{th}$ running Algorithm \ref{algorithm:gradient_descent}, the overall regret redundancy from expert mixture is bounded by $m$ times the base regret bound for the mixture of two experts, which is as follows.

\begin{lemma} \label{lemma:two-expert}
Producing each mixture decision $\decision_t^{(k)}$ results in a regret such that
\begin{displaymath}
\max\left( 
\sum_{t=1}^T g_t^\top (\decision_t^{(k-1)} - \decision_t^{(k)}), 
\sum_{t=1}^T g_t^\top (\decision_t^{(k-1)} - \decision_t^k)
\right)
\end{displaymath}
is $O\left(D \sqrt{ \sum_{t=1}^T \|\loss_t\|^2}\right)$.
\end{lemma}

This brings us to the following corollary.
\begin{corollary} \label{corollary:total-expert-regret}
From Lemma \ref{lemma:two-expert} and \eqref{eq:expert-split}, we conclude that the regret redundancy from the expert mixing is bounded as
\begin{displaymath}
R_T^{l,e}(m) \leq C_e  m D \sqrt{ \sum_{t=1}^T \|\loss_t\|^2},
\end{displaymath}
for some constant $C_e$.
\end{corollary}

\begin{theorem}
Using Corollaries \ref{corollary:mth_running_algorithm} and \ref{corollary:total-expert-regret}, the overall regret (dynamic) $\dynamicregret$ from \eqref{eq:separated_regret} can be bounded as
\begin{displaymath}
\dynamicregret \leq C_d \sqrt{D (P + D) \sum_{t=1}^T \|\loss_t\|^2},
\end{displaymath}
for some constant $C_d$.
\end{theorem}
\begin{proof}
It derives from the fact that $m \leq C_o \sqrt{2^{m-1}}$ for some constant $C_o$.
\end{proof}

Up till now, we have considered that our comparators need to come from a feasible set $\Set$ with diameter $D$, and nothing more. In the following subsections, we show that the regret has further efficiency for concentrated comparators.

\newpage
\subsection{Effective Decision Set} \label{section:domain}
Suppose the comparator sequence is concentrated, i.e. $\decision_t^*$ are located near some point $\decision^*$ in set $\Set$. To improve our regret bounds with regards to such an ``effective'' decision set $\Set_*$, we first notice that the algorithms in both our subgradient descent and expert mixing are scale and translation free as follows.

\begin{lemma} \label{lemma:set_map}
	Consider our previous algorithms, and two cases with same loss sequences but different feasible sets.
	\begin{itemize}
		\item Loss sequence is $\{\loss_t\}_{t=1}^T$.
		\item \textbf{Case 1}: Feasible set is $\Set_1$.
		\item \textbf{Case 2}: Feasible set is $\Set_2$.
	\end{itemize}
	Suppose $\Set_1$ and $\Set_2$ are one-to-one, i.e.,
	$$[\decision \in \Set_1] \iff [(\alpha(\decision-c_1) + c_2) \in \Set_2]$$ 
	for some scalar $\alpha$ and center points $c_1,c_2 \in \mathbb{R}^\dimcnt$. 
	Then, our methods output $\{\decision_t\}_{t=1}^T$ and $\{\alpha(\decision_t-c_1) + c_2\}_{t=1}^T$ if algorithms start with $\decision_1=c_1$ for $\Set_1$ and $\decision_1=c_2$ for $\Set_2$.
\end{lemma}

Keeping this in mind, consider that $\decision_t^*$ are such that $$\decision_t^* = x + \sum_{\tau=1}^M \altdec_t^{(\tau)}$$ for some natural number $M$, where $x \in \Set$ and $\altdec_t^{(\tau)} \in \Set_{\tau}$ for some $\Set_{\tau}$. Note that, for each $\tau$, $\Set_{\tau}$ and $\Set$ are related as described in Lemma \ref{lemma:set_map} with suitable $\alpha,c_1,c_2$.

Thanks to the universality of our algorithm, the regret becomes
\begin{equation*}
	O\left(
	D\sqrt{\sum_{t=1}^T \|\loss_t\|^2}
	+ \sum_{\tau=1}^{M} \sqrt{D_\tau (P_\tau + D_\tau) \sum_{t=1}^T \|\loss_t\|^2}
	\right),
\end{equation*}
where $D_\tau$ is the diameter of $\Set_\tau$, and $\sum_{\tau=1}^{M} P_\tau \geq P$ with the possibility to enforce equality when, only for a single $\tau$, $\altdec_{t}^{(\tau)}$ is allowed to change in succession.
For $M=1$, this can also be interpreted as
\begin{equation*}
	O\left(
	D\sqrt{\sum_{t=1}^T \|\loss_t\|^2}
	+ \sqrt{D_* (P + D_*) \sum_{t=1}^T \|\loss_t\|^2}
	\right).
\end{equation*} 

Thus, we have obtained dynamic guarantees, which no longer depend only on the diameter $D$ of the feasible set $K$, but also the diameter $D_*$ of the set incurred from the comparator sequence $\{\decision_t^{*}\}_{t=1}^T$. Hence, the regret result is also universal from a secondary perspective such that, for each comparator sequence generating from our feasible set, the diameter dependence is scale-free with respect to the comparator decisions, and this result is again achieved for all $\{\decision_t^{*}\}_{t=1}^T$ and not just for some ``best'' sequence. 
Furthermore, when comparing $D$ and $\sqrt{D_*P + D_*^2}$, if $D^2$ is $O(D_* (P+D_*))$, the minimax optimality is also preserved overall. Even if $D^2$ is not $O(D_* (P+D_*))$, the dependency on our actual set diameter $D$ is the optimal one for a fixed comparator in previous works, so it is unlikely to achieve better results.

\subsection{Unconstrained Optimization}
Here, we show how to conduct online convex optimization when the decision set is not necessarily bounded. For that, select a center point, e.g. the origin. Then, using the scale-free and translation-free properties of our algorithms as in Lemma \ref{lemma:set_map}, consider a mapping $\mathcal{K}(\cdot)$ from non-negative reals to feasible sets such that the set $\mathcal{K}(D)$ has diameter $D$ and center point $c$, e.g. Euclidean ball centered at $c$.

Next, we consider the following preliminary separation of the $\dynamicregret$ in \eqref{eq:linearized_regret}.
\begin{equation} \label{eq:domain_separate}
	\dynamicregret = \sum_{t=1}^T \loss_t^\top (D_t u_t - D_* u_t) + \sum_{t=1}^T \loss_t^\top(D_* u_t - D_* u_{t}^*),
\end{equation}
where $u_t,u_t^* \in \mathcal{K}(1)$, $D_t u_t = \decision_t$ and $D_* u_{t}^* = \decision_t^*$ such that $\decision_t \in \mathcal{K}(D_t)$ and $\decision_t^{*} \in \mathcal{K}(D_*)$. Manipulating \eqref{eq:domain_separate}, we get
\begin{displaymath}
	\dynamicregret = \sum_{t=1}^T (\loss_t^\top u_t) (D_t - D_*) + D_* \sum_{t=1}^T \loss_t^\top (u_t - u_{t}^*),
\end{displaymath}

Thus, we can optimize $D_t$ and $u_t$ separately. The second (latter) sum can be bounded by the techniques in our work as 
$$O\left(\sqrt{D_* (P + D_*) \sum_{t=1}^T \|\loss_t\|^2}\right).$$ 
Here, we have $D_*/2 \leq \max_{1\leq t\leq T} \|\decision_t^* - c\|$, and the path variation $P = \sum_{t=1}^{T-1} \|\decision_{t+1}^* - \decision_t^*\|$ for $\decision_t^*$ sequence as before.

The first sum is a one-dimensional linear optimization problem, and can be solved with our technique by setting path-variation $P$ to $0$, i.e. static comparator $D_*$, with the feasible set being $[0,D]$ for some known upper-bound $D$, where the regret bound component satisfy $O(D \sqrt{\sum_{t=1}^{T} \|g_t\|^2})$ since $(g_t^\top u_t)^2 \leq \|g_t\|^2$. Alternatively, $D_t$ can be solved as a $1D$ unconstrained optimization problem for unknown $D_*$, in a follow the leader manner, with the techniques available in the said literature, where an example regret bound satisfies both $\tilde{O}((D_*)^3\|\loss_t\| + \|\loss_t\|(D_*+1)\sqrt{T})$ or $\tilde{O}((D_*)^3\|\loss_t\|T^{1/3} + \|\loss_t\|D_*\sqrt{T} + \|\loss_t\|T^{1/3})$ \cite{pmlr-v99-cutkosky19a}. There are other techniques to solve this problem with differing regret results, which prioritizes other things and are not directly comparable to each other.
Even from the sum, we can see that no information regarding the dynamic nature of $\decision_{t}^*$ is carried to the domain optimization, i.e. the optimality regarding $P$ in a stand-alone manner is preserved. The work on unbounded domain optimization is still ongoing and as long as that area of work improves, so does our guarantees.

\newpage
\subsection{Finite-Per-Round Time/Computation Complexities}
\begin{figure}
	\caption{Illustrating Alternating Recursive Mixing Algorithms}
	\centering
	\includegraphics[width=\linewidth]{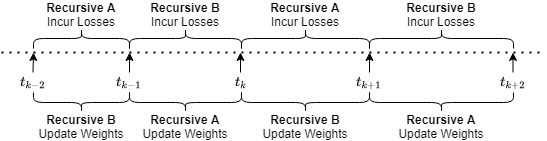}
	\label{fig:alternating}
\end{figure}

In this subsection, we investigate how to further reduce both the time and computational complexities to the $O(1)$ level. The main idea consists of running a set of two mixture frameworks (denoted as $Recursive\text{ }A$ and $Recursive\text{ }B$ respectively) using a common set of expert (run) ensemble and utilizing a wait-and-update strategy, as illustrated in \figurename \ref{fig:alternating}, which shall be explained shortly. These mixture frameworks alternate among themselves to provide the mixing probabilities $p_{t,m}$ for the generation of our final decision $\decision_t = \sum_{m=1}^{M_t} p_{t,m} \decision_t^m$. 

A major difference between the two components of our overall regret, namely $R_T^{l,e}$ and $R_T^{l,d}$, is that they are respectively static and dynamic, i.e., the comparator in one is fixed (weights) while it is time-variant (decisions) in the other. Consequently, we can follow a wait-and-update strategy for the mixture weights, i.e., wait for $\tau\leq R$ rounds and update the mixture weights afterwards using the accumulated losses in that $\tau$-length time window. This would translate into a multiplicative redundancy of at most $\sqrt{R}$ in the mixture static regret.  

This bound results from each interval $[t_{k-1}+1,t_k]$ corresponding to a wait window during which the losses $\loss_t^\top \decision_t^m$ are accumulated. Since we have $O(\log_2 t)$ agents in use, we can employ such a strategy with time-variant $\tau \in \Theta(\log_2 (t))$ for a time-frame starting at $t$. In combination with the alternating framework approach, this strategy would reduce the per-round time and computational complexities regarding expert mixture 
to $O(1)$ and would multiply $R_T^{l,e}$ by $O(\sqrt{\log_2 T})$. Framework alternation itself has only a finite effect on the regret.

\newpage
The same strategy of waiting would not work similarly for $R_T^{l,d}$ in the analyses since comparator sequence in question is not static, but dynamic, and is prone to change (however minor) during the waiting time-window. Without sacrificing deterministic regrets we had thus far, 
we can reduce only the time to $O(1)$, via parallel processing of our runs.
The computational complexity remains $O(\log_2 T)$ since we update all the essential experts at each round separately, which we need to do, as the projection operation into the feasible set $\Set$ can be a rather complicated function in the analyses, even though its computation can be rather efficient. Then, all we can do is randomly update a constant-sized subset of experts at each round, e.g. only one of them. 
The random selection can be rather cost-efficient so long as we have access to a combination of low-complexity random number generator and a look-up table.
In expected regret, since we have $\log_2 T$ experts to randomly choose among, this would result in a multiplicative redundancy of $O(\sqrt{\log_2 T})$ in $R_T^{l,d}$ and $O(\log_2 T)$ in $\left(R_T^{l,d}\right)^2$ \cite[Corollary 5]{Gokcesu}, but not $R_T^{l,e}$, since we can use all the losses incurred in mixture weight updates as previously discussed, thanks to the framework alternation and wait-and-update strategy. Even though, this regret guarantee is in expectation, since probability used to randomly select an expert is sufficiently high, i.e. $\Theta(1/\log_2 T)$, we can generate high probability guarantees using \cite{Beygelzimer} by analyzing the squared dynamic regret, i.e. $\left(R_T^{l,d}\right)^2$. After an application of $a^2 + b^2 \leq (a+b)^2$ for nonnegative values $\{a,b\}$, we incur an additive redundancy of $O(L(\log_2 T)\sqrt{\ln(1/\epsilon)})$ 
when the bound 
holds with probability at least $1-\epsilon$.

\newpage
\subsection{Lowering Memory Complexity}
\begin{figure}
\caption{Illustration of Deleting Even Runs and Nearby Mixers Joining}
\centering
\includegraphics[width=\linewidth]{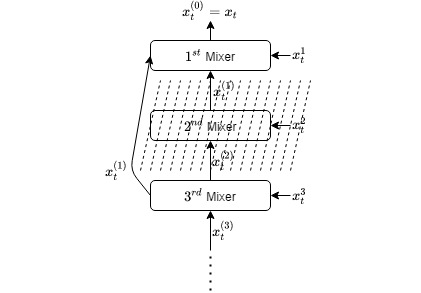}	
\label{fig:delete_evens}
\end{figure}

The regular memory usage has $O(\log_2 t)$ complexity, disregarding decision set dimension $\dimcnt$ and the like. To limit this usage with another function of $t$, i.e. $h(t)$, starting with the expert having lowest-set $P$, we would need to eliminate the even-indexed (or possibly odd) experts and loop back, re-index the runs \& repeat as needed to stay below the memory limit. With this, we can also reduce the number of cascading mixtures in effect. The procedure is illustrated in \figurename \ref{fig:delete_evens}. This would cause an increase in the dynamic regret in accordance with Corollaries \ref{corollary:gradient_descent_adaptive_eta} and \ref{corollary:mth_running_algorithm}. Let us assume in the end, between successive experts, we have at most a multiplicative discrepancy of $b$, i.e. $P_m/P_{m-1} \leq b$ for all $m$ after re-indexing of the experts following eliminations. We would come across a multiplicative redundancy of $\sqrt[4]{b}$ in the dynamic regret where the first square root is due to $P$ being inside a root in the guarantee from Corollary \ref{corollary:gradient_descent_adaptive_eta} and the second square root is thanks to us having the ability to choose among $P_{m-1}$ and $P_{m}$, which respectively lower and upper bound $P$, as the best expert in \eqref{eq:separated_regret}. Thus, 
if $b = O\left(T^{1/h(T)}\right)$ is subpolynomial, which requires infinitely growing $h(t)$ (number of expert) however slow in growth, we get subpolynomial multiplicative redundancy with respect to $T$, i.e. it is $O(T^\delta)$ for all $\delta > 0$. On the other hand, if the memory is finite, our multiplicative redundancy becomes $O(T^{0.25/h})$ for the constant memory limit, i.e. $h(T) = h$.

\newpage
\section{Simulation} \label{sec:simulation}
In this section, we simulate the performance of our approach and make relevant comparisons.
Our simulation consists of a data stream, where we want to estimate the next sample point. 

From the perspective of an automation system, or a control mechanism, this would correspond to the attempt of matching a given set of inputs to some desired outputs, which may very well act as the inputs themselves in some form (i.e., as a reference).

As shown next, our universal approach has superior performance in the problem of online convex optimization.

\subsection{Data Generation}
The specifics of our learning environment are as follows.
\begin{itemize}
	\item We have the unknown target sequence $\{y_t\}_{t=1}^T$, each of which is a two-dimensional vector.
	\item When we decide on $\decision_t$ at time $t$, we incur the $\ell^1$ error as the loss, i.e. $\|\decision_t - y_t\|_1$.
	\item There is no additional (e.g. contextual) information.
	\item For each coordinate of $y_t$, we observe whether our estimation was over or under, i.e., the sub-gradient $g_{t}$ is such that, for each dimension $k$,
	\begin{equation*}
		g_{t,k} = \begin{cases}
			  	1&  \text{ if } \decision_{t,k} > y_t,
			\\	-1&	\text{ if } \decision_{t,k} < y_t,
			\\	g&	\text{ o.w.},	
		\end{cases} \label{eq:sim_grad}
	\end{equation*}
	where $g$ can be arbitrarily located in $[-1,1]$, e.g. $g=0$.
	\item The feasible set is an origin-centered Euclidean ball such that $\|x_t\|_2 \leq (D/2)$ for some $D$, where $D$ becomes the set diameter.
	\item $y_t$ are constructed as $y_t = u_t+v_t$, where $u_t$ and $v_t$ have different dynamics.
	\item The two-dimensional vectors $u_t$ and $v_t$ are constructed as
	\begin{align*}
			&u_{t,1} = U_t \cos(\theta_t), \qquad &u_{t,2} = U_t \sin(\theta_t),
		\\	&v_{t,1} = V_t \cos(\gamma_t), \qquad &v_{t,2} = V_t \sin(\gamma_t),
	\end{align*}
	where $U_t,V_t$ are randomly selected from $[0.5,1.5]$ via the uniform distribution in an independently and identically distributed manner.
	\item The sequences of $\theta_t$ and $\gamma_t$ are subject to change in time. The number of rounds between successive changes gets progressively larger. This ensures that the quantity of change with respect to the phases $(\theta_t,\gamma_t)$ remains in $o(T)$, which makes learning possible. The distinction lies in the fact that $\gamma_t$ sequence displays a less dynamic nature. 
	\item For a change following time $t$, $\theta_{t+1}$ is selected from $[0,2\pi)$ in a uniformly random manner. For $\gamma_t$, the change is such that $\gamma_{t+1} = \gamma_t + \gamma_t'$, where $\gamma_t'$ is selected from $[0,\Gamma_t]$ in a uniformly random manner, where $\Gamma_t > 0$ also gets progressively smaller.
\end{itemize}

\subsection{Subjects of Comparison} \label{sec:comparison}
A total of five estimators are run for this estimation task, where two of them are for comparison and baseline generation.
The first is \textbf{True Oracle}, which knows the phases ($\theta_t$, $\gamma_t$) and the stable magnitudes $(\|u_t\|_2$,$\|v_t\|_2)$, which are the median $1$. Since the randomness of amplitudes $U_t$,$V_t$ occurs every round, the best candidate for the optimal competitor with sub-linear $o(T)$ path variation is this true oracle. The second one is \textbf{Last Best}, which presumes to know the last target $y_t$ and sets $\decision_{t+1} = y_t$. For the others, $y_t$ is partially known via $g_t$ \eqref{eq:sim_grad}.

The other three are the sub-gradient based learning algorithms, as explained in this work. 
The first of these is \textbf{Static}, which competes against the best fixed decision.
The second one is \textbf{Dynamic}, which presumes to know the path variation of the best possible-to-learn competitor (i.e. the true oracle) and competes against dynamic strategies, in a min-max optimal manner for the known path variation. 
The third on is \textbf{Universal}, which is the main result of this work, as in \ref{sec:no_know}.

\subsection{Performances}
\begin{figure}
	\centering
	\begin{subfigure}
		\centering
		\includegraphics[width=0.986\linewidth]{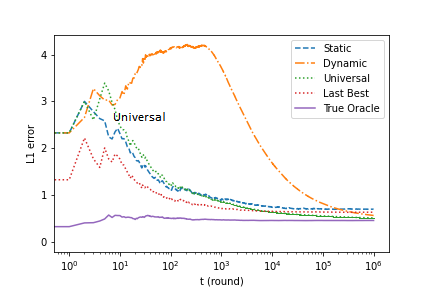}
		\label{fig:whole}
	\end{subfigure}
	\hfill
	\begin{subfigure}
		\centering
		\includegraphics[width=0.986\linewidth]{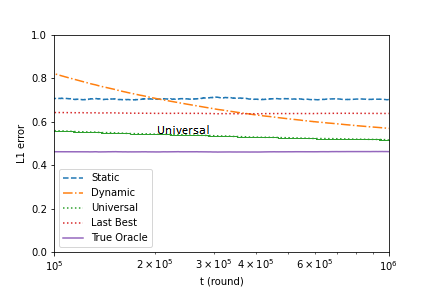}
		\label{fig:zoom}
	\end{subfigure}
	\caption{Online Estimation Performances (Linear-Log Plots)}
	\label{fig:simulation}
\end{figure}

Figure \ref{fig:simulation} (semi-log) shows the performance of each estimator in Section \ref{sec:comparison}, where the errors are cumulatively averaged. 

As we have expected, `True Oracle' performs the best, since it is effectively the best strategy with sub-linear path variation. However, the true oracle is infeasible to acquire, so here, it serves as a goal (best achievable) for the other learners. 

Until a certain round, `Dynamic' performs the worst. This can be explained with the fact that its step-size during the starting rounds becomes detrimentally large, since it uses the true path variation, which is a sizable quantity for the duration of one million rounds. Regardless, the final result demonstrates the validity of our analyses, since its performance surpasses two of the others.

The `Universal' algorithm performs the best and most robustly, as expected in Section \ref{section:domain}.

\newpage
\section{Conclusion} \label{sec:conclusion}
We first introduced an optimal sequential selection of the learning rates $\eta_t$ for the projected online subgradient descent algorithm, and achieved the minimax optimal \emph{dynamic} regret $O(\sqrt{D\left(P+D\right)}G_T)$ with the comparator path variation bounded
as $\sum_{t=1}^{T-1} \|\decision_{t+1}^* - \decision_t^*\| \leq P$, for a decision set $\Set$ with the diameter $D$ and the squared subgradient norm sums abiding by $G_T^2 = \sum_{t=1}^{T} \|g_t\|^2$. This guarantee is completely adaptive to the subgradient sequence in the minimax sense. 
We then introduced an approach to handle a time-growing $P$, i.e. $P(T)$, via resetting the learning rates at critical rounds.
This approach has resulted in a 
$\tilde{O}(1)$ multiplicative regret guarantee redundancy, all the while preserving constant-per-round computational and memory complexity. Similarly, we have also investigated the case of partial feedback on $P$ with minimax optimal guarantees.
Furthermore, we also showed the ability to distributively optimize the individual coordinates with independent runs of our algorithm and achieve minimax optimal dynamic regret guarantees for separable (e.g. hyper-rectangular) decision sets. 

Beyond these, we have further managed to eliminate the requirement of knowledge regarding the path variation bound $P$, whether a predetermined value (constant) or in the form of a time-growing function $P(T)$. We have accomplished this by transforming our initial procedure into a two-layered approach, where in one layer, we run multiple versions of the initial procedure as agents, with each agent having set a carefully selected but different possibilities for the unknown $P$, and in the other layer, we combine the decisions produced by these versions under mixture of experts setting in a recursive and cascading manner. Hence, we decreased both the overall computational and memory complexities by allowing certain orders of discrepancies between the actual $P$ and the $P_m$ set by different parallel running versions of the original procedure, i.e. within a multiplicative constant factor where some $P_m$ is less than the double and greater than the half of true $P$. This modification to our approach also comes with the property of universality such that any $P$ can be true $P$, i.e. we compete against all comparator sequences and obtain regret guarantees which are minimax optimal with respect to the path variations of each comparator sequence separately. 
Additionally, to achieve a truly online behavior with no knowledge on the time horizon $T$, we restructured these individual runs incorporating specific forms of time increasing $P(T)$, namely a piecewise combination achieved by a linearly increasing function followed by a constant indefinitely. This approach have also resulted in a
branching-out formation in the expert (agent) ensemble similar to a stairway into the higher orders of $P$ as more runs are incorporated into the ensemble at critical rounds as time goes on, to account for the higher $P$. 

Finally, we introduced different approaches to reduce time, computational and memory complexities, to the extent of constant per round, with $\widetilde{O}(1)$ multiplicative regret redundancies. Moreover, we have also displayed that universality can also be achieved in the form of our decision set (namely its diameter), i.e. the dependence of our guarantees on diameter $D$ is replaced $D_* = \max_{1\leq t\leq T} \|\decision_t^* - \decision\|$ for optimal concentration center $\decision$.

\newpage
\bibliographystyle{IEEEtran}
\bibliography{IEEEfull,bibliography}

\newpage
\appendix\label{sec:appendix}
\subsection*{Proof of Theorem \ref{theorem:gradient_descent_regret}}
	Consider the first \textit{while} loop at Line \ref{0-vec} in \figurename \ref{algorithm:gradient_descent}, which terminates at $t = t_0 \leq T$, where $t_0 \geq 1$. If the \textit{while} loop does not terminate, it means all the sub-gradients $g_t$ are zero-vectors and according to \eqref{eq:linearized_regret}, we trivially incur $0$ regret.
	
	Define $\altdec_{t+1} = \decision_t - \eta_t g_t$ such that $\decision_{t+1} = \proj{\Set}{\altdec_{t+1}}$ for $t\geq t_0$.
	Then, using \eqref{eq:linearized_regret}, we replace $g_t$ with $(1/\eta_t)(\decision_t - \altdec_{t+1})$ 
	and after rearranging the right-hand side, we get
	\begin{equation*}
	\dynamicregret \leq \sum_{t=t_0}^T \frac{1}{2\eta_t} (\|\decision_t - \decision_t^*\|^2 - \|\altdec_{t+1} - \decision_t^*\|^2) + \frac{\eta_t}{2} \|g_t\|^2,
	\end{equation*}
	since $g_t^\top (\decision_t - \decision_t^*) = 0$ for $t < t_0$.
	
	Provided that $\decision_{t+1} = \proj{\Set}{\altdec_{t+1}}$ where $\Set$ is a convex set, we have $\|\decision_{t+1} - \decision_t^*\| \leq \|\altdec_{t+1} - \decision_t^*\|$ for $\decision_t^* \in \Set$ \cite{Bubeck}. 
	
	Noting $\eta_t \geq 0$ for $t\geq t_0$, we upper bound $-\|\altdec_{t+1} - \decision_t^*\|$ with $-\|\decision_{t+1} - \decision_t^*\|$ and, for $t_0\leq t\leq T-1$, we further upper bound as
	\begin{align*}
	-\|\decision_{t+1} - \decision_t^*\|^2
	&= -\|\decision_{t+1} - \decision_{t+1}^* + \decision_{t+1}^* - \decision_t^*\|^2
	\leq -\|\decision_{t+1} - \decision_{t+1}^*\|^2 + 2D\|\decision_{t+1}^* - \decision_t^*\|,
	\end{align*}
	since $\|\decision_{t+1} - \decision_{t+1}^*\| \leq D$ where $D$ is the diameter of $\Set$ which includes all iterations $\decision_{t}$ and optimal points $\decision_{t}^*$.
	
	We also bound $-\|\decision_{T+1} - \decision_T^*\|/\eta_T\leq 0$. After regrouping
	\begin{equation*}
	\dynamicregret \leq
	\sum_{t=t_0}^T \frac{D_t^2}{2} \left(\frac{1}{\eta_t} - \frac{1}{\eta_{t-1}}\right) 
	+ D\sum_{t=t_0}^{T-1} \frac{P_t}{\eta_t}
	+ \sum_{t=t_0}^T \frac{\eta_t}{2} \|g_t\|^2,
	\end{equation*}
	where $D_t = \|\decision_t - \decision_t^*\|$, $P_t = \|\decision_{t+1}^* - \decision_t^*\|$ and $1/\eta_{t_0-1} = 0$ is a placeholder.
	
	Since $(1/\eta_t-1/\eta_{t-1}) \geq 0$ for $t \geq t_0$, we further upper bound by replacing all $D_t$ with $D$ again. This turns the first sum of the right-hand side into a telescoping sum. After we additionally bound $P_t/\eta_t \leq P_t/\eta_T$ and $\sum_{t=t_0}^{T-1} P_t \leq \sum_{t=1}^{T-1} P_t$, 
	\begin{equation*}
	\dynamicregret \leq \frac{D^2(P/D+1/2)}{\eta_T} + \sum_{t=t_0}^T \frac{\eta_t}{2} \|g_t\|^2,
	\end{equation*}
	where $P \geq \sum_{t=1}^{T-1} P_t$. This concludes the proof.

\subsection*{Proof of Corollary \ref{corollary:gradient_descent_adaptive_eta}}
	According to \eqref{eq:root_sum_g_t^2}, $G_t$ is a nondecreasing nonnegative sequence. Until $G_t > 0$ for some $t=t_0$, we incur $0$ regret. Afterwards, for $t \geq t_0$, $\eta_t$ becomes a nonincreasing positive sequence. Thus, we can build upon the result of Theorem \ref{theorem:gradient_descent_regret}.
	
	From \eqref{eq:root_sum_g_t^2}, $\|g_t\|^2 = G_t^2 - G_{t-1}^2$ where $t \geq 1$. Combined with Theorem \ref{theorem:gradient_descent_regret} and "difference of two squares",
	\begin{equation*}
	\dynamicregret \leq 
	\frac{D^2(P/D+1/2)}{\eta_T} + \sum_{t=t_0}^T \frac{\eta_t}{2} (G_t - G_{t-1})(G_t + G_{t-1}).
	\end{equation*}
	As $\eta_t$'s are positive and $G_t$'s are nondecreasing, we can upper-bound right-hand side by replacing $(G_t + G_{t-1})$ with $2G_t$. Then, we put in $\eta_t = DG_t^{-1}\sqrt{\hat{P}/D+1/2}$ and obtain a telescoping sum. After we also bound $-G_{t_0}$ with $0$, we arrive at the corollary.
	
\subsection*{Proof of Claim \ref{claim}}
Denote $A_T = \sum_{t=1}^T g_t g_t^T$. Consequently,
\begin{equation} \label{eq:trace}
	tr(A_T) = \sum_{t=1}^{T} tr\left(g_t g_t^T\right) = \sum_{t=1}^{T} g_t^T g_t = \sum_{t=1}^{T} \|g_t\|^2,
\end{equation}
since trace is a linear operation, is equivalent to summing the eigenvalues, and only nonzero eigenvalue of $g_t g_t^T$ is $g_t^T g_t$.

Denote the eigenvalues of $A_T$ as $\lambda_1, \ldots, \lambda_\dimcnt$. 
We also note that positive semidefinite matrices $g_t g_t^T$ sum to $A_T$. Therefore, $A_T$ is also a positive semidefinite matrix where $\lambda_1,\ldots,\lambda_N$ are all nonnegative. Consequently, the square root operation on the symmetric $A_T$ effectively replaces the eigenvalues with their square roots. This implies
\begin{equation} \label{eq:duchis}
	tr\left(\left[\sum_{t=1}^T g_t g_t^T\right]^{1/2}\right) = \sum_{i=1}^N \sqrt{\lambda_i}.
\end{equation}
Additionally, we have 
$
	\sqrt{\sum_{t=1}^T \|g_t\|^2} = \sqrt{\sum_{i=1}^N \lambda_i}.
$
from \eqref{eq:trace} due to the trace operation. 
Comparing the squares of this and \eqref{eq:duchis} while noting that $\lambda_i\geq 0$, we arrive at the claim.

\newpage
\subsection*{Proof of Theorem \ref{theorem:lower_bound_sum}}
	To show that some sequence of functions $\{f_t(\cdot)\}_{t=1}^T$ exists -in accordance with our decisions- which results in at least the worst-case regret lower bound in hindsight as claimed in this theorem, we are free to restrict our analysis to linear functions, i.e. $f_t(\cdot) = g_t^\top (\cdot)$, and show that, in hindsight, a function sequence consisting of only such linear functions exists and satisfies the intended lower bound in the theorem. This can be thought of as easing the analysis by further lower bounding. Moreover, to prove the existence of such a function sequence, we only need to take expectation of the said regret lower bound over our candidate sequences with respect to some distribution, both of which are chosen by us in a purposeful manner. This step can also be thought of as further lower bounding. One thing to note is that both of these "lower bounding" steps do not loosen the bounds much, as the results are shown to match with the regret guarantees (upper bounds) we have previously generated, up to a constant factor. Considering these, the worst-case dynamic regret $\overline{\dynamicregret}$ satisfies
	\begin{equation*}
	\overline{\dynamicregret} \geq \E{}{\sum_{t=1}^{T} g_t^\top \decision_{t} - \min_{\{\decision_t^*\}_{t=1}^T \in \Omega_T^P} \sum_{t=1}^T g_t^\top \decision_{t}^*},
	\end{equation*}
	over some expectation for $g_t$. We lower bound further by restricting $\{\decision_t^*\}_{t=1}^T \in \Omega_T^P$ such that it remains constant at certain time intervals, i.e. $\decision_t^* = \decision_{(k)}^* \in \Set$ for $t_{k-1} < t\leq t_k$ where $1\leq k\leq \lfloor P/D \rfloor + 1$ and $t_0=0$, $t_{\lfloor P/D \rfloor + 1}=T$. This is a valid lower bounding as we effectively shrink the search space of $-\min(\cdot)$ operation and it is fully encapsulated by $\Omega_T^P$ as $P \geq D(\lfloor P/D \rfloor)$. This results in a further lower bound
	\begin{align*}
	\overline{\dynamicregret} \geq &\E{}{\sum_{k=1}^{\lfloor P/D \rfloor + 1}\left(\sum_{t=t_{k-1}+1}^{t_k} g_t^\top \decision_{t} - \sum_{t=t_{k-1}+1}^{t_k} g_t^\top \decision_{(k)}^*\right)},
	\\&\text{ where }
	\decision_{(k)}^* = \argmin{\decision \in \Set} \sum_{t=t_{k-1}+1}^{t_k} g_t^\top \decision.
	\end{align*}

	This worst-case regret lower bound can be simplified via an analysis similar of which can be found in \cite[Appendix F]{Orabona} and shall not be repeated here explicitly. The analysis includes, by assuming "the worst", selecting $g_t$ such that they have norms $\|g_t\|$ and they are parallel to each other and the line joining any two points in $\Set$ farthest from each other. The direction of each $g_t$ is uniformly randomly selected using independent Rademacher distributions out of the two possible directions. The analysis concludes with the use of Khinchin's inequality separately for each stationary $\decision_{(k)}^*$, which gives
	\begin{equation*}
	\overline{\dynamicregret} \geq \sum_{k=1}^{\lfloor P/D \rfloor + 1} \frac{D}{2\sqrt{2}} \sqrt{\sum_{t=t_{k-1}+1}^{t_k} \|g_t\|^2}.
	\end{equation*}
	Then, there exists a scenario where $g_t$ sequence is such that $\sum_{t=t_{k-1}+1}^{t_k} \|g_t\|^2 = G_T^2/(\lfloor P/D \rfloor + 1)$. This is a rather mild assumption as we assume the setting to be adversarial, i.e. free in its choice of $g_t$, and, furthermore, $t_k$ are selected by the adversary (environment/setting). 
	
	Thus, even in the presence of Lipschitz constraints, i.e., $\|g_t\|\leq L_t$, for sufficiently large $T$ and relatively low (e.g. finite) $\|g_t\|, L_t$, the adversary can enforce this assumption with arbitrarily small error. Consequently, replacing each summand in the lower bound with this mild assumption gives the bound in this theorem.

\subsection*{Proof of Corollary \ref{corollary:lower_bound_max}}
	We similarly restrict $\xts$ as in the proof of Theorem \ref{theorem:lower_bound_sum} such that it remains constant inside distinct time intervals. Then, 
	in a scenario which is possibly the "worst-case", $\{g_t\}_{t=1}^T$ sequence can be such that, for each $k$, we have $\sum_{t=t_{k-1}+1}^{t_k} \|g_t\|^2 \geq L^2 \lfloor T/(\lfloor P/D \rfloor + 1) \rfloor$, 
	since 
	$t_k$ can be freely chosen by the adversarial setting which ensures that, when the time segments are of equal length, $t_k - t_{k-1}$ can at least be $\lfloor T/(\lfloor P/D \rfloor + 1) \rfloor$ for a total of $(\lfloor P/D \rfloor + 1)$ segments. Since $\lfloor P/D \rfloor + 1 \leq T$ and each segment is at least of length $1$, we also get $\sum_{t=t_{k-1}+1}^{t_k} \|g_t\|^2 \geq L^2 (T/2)/(\lfloor P/D \rfloor + 1)$ which results in the corollary similar to Theorem \ref{theorem:lower_bound_sum}.
	
\subsection*{Proof of Theorem \ref{theorem:with_resets}}
	The regret inequality directly follows from Corollary \ref{corollary:gradient_descent_adaptive_eta} since the algorithm reset for durations corresponding to each $P_k$ which are known prior to the times they are needed for each respective run. 
	This is a tight bound due to the following analysis.
	
	Consider the sum $\sum_{k=1}^{K} \sqrt{a_k x_k}$ with nonnegative constants $a_k$'s and variables $x_k$'s under the constraint 
	$$\sum_{k=1}^{K} x_k = X.$$ This sum is concave with respect to the variables $\{x_k\}_{k=1}^K$ and is maximized when $x_k = s\, a_k$ for some $s$. By summing both sides of this equation over all $k$, we see that $s = X/\sum_{k=1}^{K}a_k$. Then, when we substitute for $x_k$ and, following that, also substitute for $s$ in the sum, it appears that
	\begin{displaymath}
	\sum_{k=1}^{K} \sqrt{a_k x_k} \leq \sqrt{\sum_{k=1}^{K}a_k} \sqrt{\sum_{k=1}^{K}x_k},
	\end{displaymath}
	after we resubstitute for $X$ with $\sum_{k=1}^{K} x_k$. Back to the problem at hand, when we set $a_k = P_k/D + 1/2$ and $x_k = \sum_{t=t_{k-1}+1}^{t_k} \|g_t\|^2$, we arrive at the theorem.

\subsection*{Proof of Theorem \ref{theorem:constraint_grows_in_time}}
	We employ an analysis similar to the proof of Theorem \ref{theorem:with_resets}. Then, the resets result in the regret
	\begin{equation*}
	\dynamicregret 
	\leq 2D\sqrt{\sum_{k=1}^{K} 
		(2^{k-1}-1) + \frac{K}{2}} \sqrt{\sum_{t=1}^{T} \|g_t\|^2},
	\end{equation*}
	where, we know, $P(T) > D(2^{K-2} - 1)$. After bounding with $P(T)$ and further arrangements, we arrive at our result. The optimality claim is apparent after noticing that $K=2$ maximizes the gap since, when $K=1$, the disparity (ratio) between $P(T)/D$ and its floor cannot exceed $1$.
	
\newpage
\subsection*{Proof of Lemma \ref{lemma:two-expert}}
	The losses used in the mixture from \cite{Cesa-Bianchi_improved} are
	\begin{displaymath}
		l_{t,m} \define \loss_t^\top (\decision_t^m) \text{ and } l_{t,(m)} \define \loss_t^\top (\decision_t^{(m)}).
	\end{displaymath}
	The regret component from the mixture reduces to 
	$$O\left(\sqrt{\sum_{t=1}^T \sum_{m' \in \{m,(m)\}} p_{t,m'} \left(l_{t,m'} - u_{t,(m-1)}\right)^2} \right)$$
where $u_{t,(m-1)} = \min(l_{t,m},l_{t,(m)})$ and $\{p_{t,m}, p_{t,(m)}\}$ are the mixture weights which sum to $1$. This bound is achieved after modifying the mixing method in \cite{Cesa-Bianchi_improved} by considering only one-sided losses, hence eliminating the need for an additional redundancy to upper-bound $\exp(x)$ where $x>0$ and having simplified mixture learning rates. One notices that both $(l_{t,m} - u_{t,(m-1)})^2$ and $(l_{t,(m)} - u_{t,(m-1)})^2$ are upper-bounded by $D\|\loss_t\|^2$ since $u_{t,(m-1)} = \loss_t^\top x$ for some $x \in K$. Thus, we obtain the bound in this lemma.

\end{document}